\documentclass[12pt]{article}
\usepackage{amsmath}
\usepackage{amssymb}
\usepackage{amsthm}
\providecommand{\abs}[1]{\lvert#1\rvert}

\newcommand{\ov}{\overline}
\newcommand{\cl}{{\cal L}}
\newcommand{\ra}{\rightarrow}
\newcommand{\dist}{\mbox{\rm dist\/}}
\newcommand{\eps}{\varepsilon}

\newtheorem{thm}{Theorem}
\newtheorem{lem}{Lemma}


\begin{document}

\begin{center}
{\bf\Large Number of lines in hypergraphs}\\
\vspace{0.5cm}
 Pierre Aboulker (Concordia University, Montreal)\footnote{{\tt pierreaboulker@gmail.com}}\\
 Adrian Bondy (Universit\' e Paris 6)\footnote{{\tt adrian.bondy@sfr.fr}}\\
 Xiaomin Chen (Shanghai Jianshi LTD)\footnote{{\tt gougle@gmail.com}}\\
 Ehsan Chiniforooshan (Google, Waterloo)\footnote{{\tt chiniforooshan@alumni.uwaterloo.ca}}\\
 Va\v sek Chv\' atal (Concordia University, Montreal)\footnote{{\tt chvatal@cse.concordia.ca}\\
\hphantom{lllllxxx}Canada Research Chair in Discrete Mathematics}\\
 Peihan Miao (Shanghai Jiao Tong University)\footnote{{\tt sandy656692@gmail.com}}\\
\end{center}

\begin{center}
{\bf Abstract}
\end{center}
{\small Chen and Chv\' atal introduced the notion of lines in
  hypergraphs; they proved that every $3$-uniform hypergraph with $n$
  vertices either has a line that consists of all $n$ vertices or else
  has at least $\log_2 n$ distinct lines. We improve this lower bound
  by a factor of $2-o(1)$.}
\vspace{0.5cm}

A classic theorem in plane geometry asserts that every noncollinear
set of $n$ points in the plane determines at least $n$ distinct lines.
As noted by Erd\H os~\cite{E43} in 1943, this is a corollary of the
Sylvester-Gallai theorem (which asserts that for every noncollinear
set $V$ of finitely many points in the plane, some line goes through
precisely two points of $V$); it is also a special case of a
combinatorial theorem proved by De Bruijn and Erd\H os~\cite{DE48} in
1948.  In 2006, Chen and Chv\' atal~\cite{CC} suggested that this
theorem might generalize to all metric spaces. More precisely, line
$\ov{uv}$ in a Euclidean space can be characterized as
\begin{multline*}
\ov{uv}=\{p:
\dist(p,u)\!+\!\dist(u,v)\!=\!\dist(p,v)\;\mbox{{\rm or}}\;\\
\dist(u,p)\!+\!\dist(p,v)\!=\!\dist(u,v)\;\mbox{{\rm or}}\;  
\dist(u,v)\!+\!\dist(v,p)\!=\!\dist(u,p)\}, 
\end{multline*}
where $\dist$ is the Euclidean metric; in an arbitrary metric space
$(V,\dist)$, the same relation may be taken for the definition of line
$\ov{uv}$.  With this definition of lines in metric spaces, Chen and
Chv\'{a}tal asked:
\begin{itemize}
\item[($\star$)]
{\em True or false? Every metric space on $n$ points, where $n\ge 2$,
  either has at least $n$ distinct lines or else has a line that
is universal in the sense of consisting of all $n$ points.}
\end{itemize}
There is some evidence that the answer to ($\star$) may be `true'. For
instance, Kantor and Patk\' os~\cite{KP} proved that
\begin{itemize} 
\item if no two of $n$ points ($n\ge 2$) in the plane share
  their $x$- or $y$-coordinate, then these $n$ points with the $L_1$
  metric either induce at least $n$ distinct lines or else they induce
  a universal line.
\end{itemize}
(For sets of $n$ points in the plane that are allowed to share their
coordinates, \cite{KP} provides a weaker conclusion: these $n$ points
with the $L_1$ metric either induce at least $n/37$ distinct lines or
else they induce a universal line.)  Chv\'{a}tal
\cite{Chv12} proved that
\begin{itemize}
\item every metric space on $n$ points where $n\ge 2$ and each nonzero
  distance equals $1$ or $2$ either has at least $n$ distinct lines or
  else has a universal line.
\end{itemize}
Every connected undirected graph induces a metric space on its vertex
set, where $\dist(u,v)$ is the familiar graph-theoretic distance
between vertices $u$ and $v$ (defined as the smallest number of edges
in a path from $u$ to $v$). It is easy to see that
\begin{itemize}
\item every metric space induced by a connected bipartite graph on $n$
  vertices, where $n\ge 2$, has a universal line.
\end{itemize} 
A {\em chordal graph\/} is a graph that contains no induced cycle of
length four or more.  Beaudou, Bondy, Chen, Chiniforooshan,
Chudnovsky, Chv\'{a}tal, Fraiman, and Zwols~\cite{BBC2} proved that
\begin{itemize}
\item every metric space induced by a connected chordal graph on $n$
  vertices, where $n\ge 2$, either has at least $n$ distinct lines or
  else has a universal line.
\end{itemize}
Chiniforooshan and Chv\'{a}tal \cite{CC11}
proved that
\begin{itemize}
\item every metric space induced by a connected graph on $n$
  vertices either has $\Omega(n^{2/7})$ distinct lines or else has a 
  universal line. 
\end{itemize}

A {\em hypergraph\/} is an ordered pair $(V,H)$ such that $V$ is a set
and $H$ is a family of subsets of $V$; elements of $V$ are the {\em
  vertices\/} of the hypergraph and members of $H$ are its {\em
  hyperedges;\/} a hypergraph is called {\em $k$-uniform\/} if each of
its hyperedges consists of $k$ vertices. The definition of lines in a
metric space $(V,\dist)$ that was our starting point depends only on
the $3$-uniform hypergraph $(V,H)$ where
$H=\{\{a,b,c\}:\,\dist(a,b)+\dist(b,c)=\dist(a,c)\}$: we have
\[
\overline{uv}\;=\;\{u,v\}\cup\{p:\{u,v,p\}\in H\}.
\]
Chen and Chv\'{a}tal~\cite{CC} proposed to take this relation for the
definition of line $\ov{uv}$ in an arbitrary $3$-uniform hypergraph
$(V,H)$. With this definition, the combinatorial theorem of De Bruijn
and Erd\H os~\cite{DE48} can be stated as follows:
\begin{itemize}
\item if no four vertices in a $3$-uniform hypergraph carry two or
  three hyperedges, then, except when one of the lines in this
  hypergraph is universal, the number of lines is at least the number
  of vertices and the two numbers are equal if and only if the
  hypergraph belongs to one of two simply described families.
\end{itemize}
Beaudou, Bondy, Chen, Chiniforooshan, Chudnovsky, Chv\'{a}tal,
Fraiman, and Zwols~\cite{BBC1} generalized this statement by allowing
any four vertices to carry three hyperedges:
\begin{itemize}
\item if no four vertices in a $3$-uniform hypergraph carry two
  hyperedges, then, except when one of the lines in this hypergraph is
  universal, the number of lines is at least the number of vertices
  and the two numbers are equal if and only if the hypergraph belongs
  to one of three simply described families.
\end{itemize}
In particular, if the `metric space' in ($\star$) is replaced by
`$3$-uniform hypergraph where no four vertices carry two hyperedges',
then the answer is `true'. Without the assumption that no four
vertices carry two hyperedges, the answer is `false'~\cite[Theorem
  3]{CC}: there are arbitrarily large $3$-uniform hypergraphs where no
line is universal and yet the number of lines is only
$\exp(O(\sqrt{\ln n}))$. Nevertheless, even this variation on
($\star$) can be answered `true' if the desired lower bound on the
number lines is weakened enough~\cite[Theorem 4]{CC}:
\begin{itemize}
\item Every $3$-uniform hypergraph with $n$ vertices either has at
  least $\lg n + \frac{1}{2}\lg\lg n + \frac{1}{2}\lg\frac{\pi}{2}-
  o(1)$ distinct lines or else has a universal line.
\end{itemize}
(We follow the convention of letting $\lg$ stand for the logarithm to
base $2$.) The purpose of our note is to improve this lower bound by a
factor of $2-o(1)$.

All our hypergraphs are $3$-uniform. We let $V$ denote the vertex set,
we let $\cl$ denote the line set, and we write $n=\abs{V}$,
$m=\abs{\cl}$.  The number of hyperedges, which we call {\em
  hedges,\/} is irrelevant to us.  We assume throughout that no line
is universal.

Let us define mappings $\alpha,\beta: V\ra 2^\cl$ by 
\[
\alpha(x)=\{L\in\cl :x\in L\} 
\;\;\text{ and }\;\; 
\beta(x)=\{\overline{xw}: w\ne x\}.
\]
Note that $\beta(x)\subseteq \alpha(x)$ for all $x$.  The proof of the
lower bound 
\begin{equation}\label{eq.orig}
m\ge \lg n
\end{equation} 
in \cite[Theorem 4]{CC} relies on the
observation that $\alpha$ is one-to-one. This observation generalizes
as follows:
\begin{lem}\label{lem.anti} If $f: V\ra 2^\cl$ is a mapping such that 
  $\beta(x)\subseteq f(x)\subseteq \alpha(x)$ for all $x$, then $f$ is
  one-to-one and $\{f(x): x\in V\}$ is an antichain.
\end{lem}
\begin{proof}
  We only need prove that $\beta(x)-\alpha(y)\ne \emptyset$ whenever
  $x\ne y$. To do this, we use the assumption that $\overline{xw}$ is
  not universal: there is a point $z$ such that
  $z\not\in\overline{xy}$. This means that $\{x,y,z\}$ is not a hedge,
  and so $\overline{xz}\in \beta(x)-\alpha(y)$.
\end{proof}
\begin{lem}\label{lem.new}
  If $x,y,z$ are vertices such that $\overline{xy} = \overline{xz}$,
  then $\alpha(y)\cap\beta(x)=\alpha(z)\cap\beta(x)$.
\end{lem}
\begin{proof} 
  If $y\in \overline{xw}$, then $\{x,w,y\}$ is a hedge, and so $w\in
  \overline{xy}=\overline{xz}$, and so $\{x,z,w\}$ is a hedge, and so
  $z\in \overline{xw}$.
\end{proof} 
We define the {\em span} of a subset $S$ of $V$ to be $\cup_{x\in S}\beta(x)$.
\begin{lem}\label{lemma1} 
  If $n\ge 2$ and a nonempty set of $s$ vertices has a span of $t$ lines, then 
\[
m-t \ge \lg(n-s)-s\lg t.
\]
\end{lem}
\begin{proof}
  Given a nonempty set of $s$ vertices and its span $T$ of $t$ lines,
  enumerate the vertices in $S$ as $x_1,x_2,\ldots ,x_s$. 
Note that $t>0$ (since $n\ge 2$ and $s>0$) and define a mapping
$\psi:(V-S)\ra T^s$ by
\[
\psi(v)=(\overline{x_1v}, \overline{x_2v}, \ldots ,\overline{x_sv}).
\]
If $y,z$ are vertices in $V-S$ such that $\psi(y)=\psi(z)$, then
Lemma~\ref{lem.new} guarantees that
$\alpha(y)\cap\beta(x_i)=\alpha(z)\cap\beta(x_i)$ for every $x_i$ in
$S$ and so (since $T=\cup_{i=1}^s\beta(x_i)$) $\alpha(y)\cap
T=\alpha(z)\cap T$.  This and Lemma~\ref{lem.anti} (with $f=\alpha$)
together imply that $\alpha(y)-T\ne \alpha(z)-T$ whenever
$\psi(y)=\psi(z)$ and $y\ne z$. It follows that $\abs{C}\le 2^{m-t}$
for every subset $C$ of $V-S$ on which $\psi$ is constant.  Since at
least one of these sets $C$ has at least $(n-s)/t^s$ points, we 
conclude that $(n-s)/t^s\le 2^{m-t}$.
\end{proof}

\begin{lem}\label{lem.tail}
For every positive $\eps$, there is a positive $\delta$ such that
\[
\textstyle{ \sum_{i<\delta N}\binom{N}{i} \le 2^{\eps N} } \;\;
\text{  for all positive integers $N$.}
\]
\end{lem}
\begin{proof} A special case of an inequality proved first by Bernstein~\cite{Ber24,Hoe63} asserts that
\[
\sum_{i=0}^k\binom{N}{i} \le \left(\frac{N}{k}\right)^{k}\left(\frac{N}{N-k}\right)^{N-k}
\;\;\text{ for all $k=0,1,\ldots ,\lfloor N/2\rfloor$;}
\]
we have 
\[
\left(\frac{N}{k}\right)^{k}\left(\frac{N}{N-k}\right)^{N-k}
\;\le \; \left(\frac{eN}{k}\right)^{k}
\;\text{ and }\; \lim_{\delta\ra 0+}\left(\frac{e}{\delta}\right)^{\delta}=1.
\]
\end{proof}

\begin{thm}\label{thm.lb}
$m\ge (2-o(1))\lg n$.
\end{thm}

\begin{proof}
Given any positive $\eps$, we will prove that 
$m\ge (2-4\eps)\lg n$ for all sufficiently large $n$. 
To do this, let $\delta$ be as in 
Lemma~\ref{lem.tail} and consider 
a largest set $S$ of vertices whose span $T$ has at least
$(0.5\delta \lg n)\cdot\abs{S}$ lines (this $S$ may be empty). Writing 
$s=\abs{S}$ and $t=\abs{T}$, we may assume that
\[
t<2\lg n
\]
(else we are done since $m\ge t$),  and so $s<4/\delta$. Now 
\[
m-t\ge (1-o(1))\lg n:
\]
this follows from Lemma~\ref{lemma1} when $t>0$ and from
\eqref{eq.orig} when $t=0$. In turn, we may assume that 
\[
t\le 0.5m
\]
(else $0.5m > m-t \ge (1-o(1))\lg n$ and we are done).  Finally,
consider a largest set $R$ of vertices such that $\beta(y)\cap
T=\beta(z)\cap T$ whenever $y,z\in R$ and note for future reference
that $\abs{R}\ge n/2^{\,t}$. Since $\beta$ is one-to-one
(Lemma~\ref{lem.anti}), all the sets $\beta(y)-T$ with $y\in R$ are
distinct; by maximality of $S$, each of them includes less than
$0.5\delta \lg n$ lines (else $y$ could be added to $S$); it follows
that (when $n$ is large enough to make $0.5\lg n$ less than $m-t$)
\[
\textstyle{ \abs{R}\le \sum_{i<0.5\delta \lg n}\binom{m-t}{i}} 
\le \sum_{i<\delta(m-t)}\binom{m-t}{i}
\le 2^{\,\eps (m-t)} \le 2^{\,\eps m},
\]
and so 
\[
n \le 2^{\,t}\abs{R} \le 2^{\,t+\eps m} \le 2^{\,(0.5+\eps) m} \le 2^{\,m/(2-4\eps)}.
\]
\end{proof}

{\bf\Large Acknowledgment}

\bigskip

\noindent The work whose results are reported here began at a workshop
held at Concordia University in April 2013.  We are grateful to the
Canada Research Chairs program for its generous support of this
workshop. We also thank Laurent Beaudou, Nicolas Fraiman, and Cathryn
Supko for their stimulating conversations during the workshop.

\end{document}